\documentclass[12pt]{amsart}
\usepackage{float}
\usepackage{amsfonts}
\usepackage{amsmath}
\usepackage{amssymb}
\usepackage{graphicx}
\usepackage{amscd}
\usepackage{xcolor}
\usepackage{hyperref}
\usepackage{tikz-cd}
\usepackage{multirow}
\DeclareMathOperator\arctanh{arctanh}
\newtheorem{theorem}{Theorem}

\newtheorem{remark}[theorem]{Remark}

\newtheorem{lemma}[theorem]{Lemma}

\newtheorem{proposition}[theorem]{Proposition}

\newcommand{\sh}{\sinh}

\title[Explicit harmonic and wave maps into variable-curvature surfaces]%
      {Explicit harmonic and wave maps into variable-curvature surfaces}
\author{A. Fotiadis, G. Polychrou}

\begin{document}


\begin{abstract}
Explicit harmonic and wave maps are typically available only in
highly symmetric or constant-curvature settings, where additional
symmetry or integrability structures are present. We develop a
reduction framework for pseudo-Riemannian surfaces that extends
explicit constructions to a geometrically significant class of
variable-curvature targets. For target metrics of the form $A(R)\,dR^2 - \delta^2 B(R)\,dS^2$, a
geometrically adapted travelling-wave ansatz reduces the
Euler--Lagrange system to a solvable system of first-order ODEs. The method applies simultaneously to harmonic and wave maps, treating the elliptic and hyperbolic regimes uniformly within a single framework. As concrete applications, we construct explicit harmonic maps into ellipsoids, Lorentzian wave maps into hyperboloids and the Schwarzschild exterior, and a mixed-signature example, all in genuinely variable-curvature geometries where explicit constructions are substantially less accessible.
\end{abstract}

\maketitle

\section{Introduction}

Although the qualitative theory of harmonic and wave maps is highly
developed, the construction of \emph{explicit} solutions remains a
difficult and largely open problem, except in very symmetric settings.
Explicit formulas are valuable because they provide concrete test
cases for analytic theory, illuminate underlying geometric mechanisms
and often reveal hidden structures that are otherwise inaccessible.
Outside a small number of highly structured situations, however, the
harmonic map equations quickly become too nonlinear to solve in closed
form.

Harmonic maps have played a central role in geometric analysis since
the work of Eells and Sampson \cite{EellsSampson,EellsLemaire}; in two dimensions,
their conformal invariance links them to complex analysis,
Teichm\"{u}ller theory and geometric structures on surfaces; see,
for example, \cite{Helein, M1, M2, Minsky, S-Y, Wolf1, Wolf2, Wolf3}
and the references therein.

Most known explicit constructions in two dimensions rely on strong
symmetry assumptions on the target, most notably constant curvature
or symmetric-space structures. In these cases, the harmonic map
equations admit additional integrability properties, such as
reductions to classical integrable systems. (\cite{FordyWood} is a reference for the link between harmonic maps and integrable systems) When the curvature of the
target varies, these structures typically break down, and the methods
developed for constant-curvature targets no longer apply.

The primary geometric motivation for exploring the present setup lies
in the interplay between variable curvature and the breakdown of
classical integrability. While constant-curvature targets often admit
additional algebraic or symmetry structures that facilitate explicit
integration, these mechanisms typically disappear in variable-curvature
geometries. Our strategy is therefore to identify a geometrically
natural class of target metrics for which the harmonic/wave map
equations still admit a meaningful symmetry reduction.

The travelling-wave ansatz used below should not be viewed merely as an
ad hoc computational device. Rather, it corresponds to a one-parameter
reduction in which the map profile evolves along the characteristic
combination $t=ay-bx$, while the second target component retains an
affine dependence on the original coordinates. This structure preserves
nontrivial coupling between the target variables while reducing the
Euler--Lagrange system to a tractable first-order problem.

Specifically, consider domain metrics of the form
\[
g = e^{f(x,y)} \bigl( dx^{2} - \epsilon^{2} dy^{2} \bigr),
\qquad \epsilon \in \{1,i\},
\]
on a surface $M$, while on the target surface $N$, we identify a geometrically significant class of
variable-curvature metrics for which the harmonic and wave map
equations admit a systematic reduction to ordinary differential
equations. For target metrics of the form
\[
h = A(R)\,dR^{2} - \delta^{2} B(R)\,dS^{2},
\qquad \delta \in \{1,i\},
\]
where $A$ and $B$ are smooth positive functions, the harmonic map
equations admit a partial separation of variables under a
travelling-wave ansatz, leading to a reduction to integrable
first-order equations.

We emphasise that the present work is mathematically distinct from our
previous contributions
\cite{FotDask1,FotDask2,PPFD}. Those works focused primarily on
harmonic \emph{diffeomorphisms} between surfaces of constant curvature,
with the Beltrami equation serving as the principal analytical tool.
The present paper differs in several essential respects.

First, we consider harmonic and wave maps that are not required to be
diffeomorphisms. Second, the target geometries considered here have
genuinely variable curvature, including ellipsoidal, hyperboloidal and
Schwarzschild-type geometries, rather than constant-curvature model
spaces. Third, the analytical mechanism is fundamentally different:
instead of a Beltrami-equation approach, we exploit first integrals of
the Euler--Lagrange system combined with a symmetry reduction that
applies uniformly in both elliptic and hyperbolic signature settings.

To the best of our knowledge, explicit harmonic and wave maps into
variable-curvature pseudo-Riemannian surfaces of this type remain
comparatively unexplored.

The main result (Theorem~\ref{Theorem1}) provides a reduction
principle for harmonic and wave maps into pseudo-Riemannian surfaces
with metrics of the above form: under the symmetry ansatz introduced
here, the coupled Euler--Lagrange system reduces to a solvable
first-order equation for the one component, while the remaining
component is recovered by quadrature. The subsequent examples
demonstrate how this mechanism produces explicit maps in several
geometrically distinct variable-curvature settings.

A particularly noteworthy application is the construction of harmonic
maps into an \emph{ellipsoid}. Ellipsoids provide canonical
variable-curvature deformations of the round sphere, yet explicit
harmonic maps into them are extremely scarce because the integrability
mechanisms available in symmetric spaces are typically absent.
Recently, \cite{BrandingSiffert} studied harmonic \emph{self}-maps of
ellipsoids in arbitrary dimensions; by contrast, we construct explicit
harmonic maps from a Riemannian surface \emph{into} a non-round
ellipsoidal target in $\mathbb{R}^{3}$.

Equally significant is the Lorentzian setting. When both domain and
target carry indefinite metrics, our framework yields explicit
wave maps into hyperboloidal geometries and into the Schwarzschild
exterior. These examples illustrate that the reduction method applies
not only to purely geometric model spaces but also to pseudo-Riemannian
backgrounds of direct relevance in mathematical physics, \cite{Oneil}.

The paper is organised as follows. Section~2 establishes notation
and derives the harmonic map equations. Section~3 contains the proof
of Theorem~\ref{Theorem1}. Section~\ref{sec:examples} presents
illustrative examples. Section~5 gives concluding remarks.

\section{Preliminaries}

In this section we establish the notation and geometric setup used
throughout the paper. Let $u : (M,g) \longrightarrow (N,h)$ be a
map between pseudo-Riemannian surfaces, written $u = (R,S)$ in local
coordinates.

\subsection{Isothermal coordinates and metric conventions}

On the domain surface $M$ we choose isothermal coordinates
\[
v = x + \epsilon y, \qquad w = x - \epsilon y,
\]
where $\epsilon \in \{1,i\}$: the choice $\epsilon = 1$ corresponds
to a Lorentzian surface and $\epsilon = i$ to a Riemannian one. In
these coordinates the metric takes the conformal form
\[
g = e^{f(x,y)} \left( dx^{2} - \epsilon^{2} dy^{2} \right)
= e^{f(v,w)}\, dv\,dw, \qquad \epsilon \in \{1,i\}.
\]

On the target surface $N$ we consider metrics of the form
\[
h = A(R)\, dR^{2} - \delta^{2} B(R)\, dS^{2}, \qquad \delta \in \{1,i\},
\]
where $A(R)$ and $B(R)$ are smooth, positive functions. This class
includes a wide range of classical variable-curvature geometries,
such as surfaces of revolution (catenoids, paraboloids, tori), as
well as Lorentzian models arising in general relativity and cosmology.
Tables~\ref{tab:2d-metrics-R} and~\ref{tab:2d-metrics-R-nonconst}
list representative examples; they are intended as a
non-exhaustive illustration of the breadth of the framework.

\begin{table}[h]
  \centering
  \renewcommand{\arraystretch}{1.25}
  \begin{tabular}{ll}
    \hline
    \textbf{Model / Example} & \textbf{Metric Form} \\
    \hline
    Flat plane (Cartesian)   & $ds^2 = dR^2 + dS^2$ \\
    Flat plane (Polar)       & $ds^2 = dR^2 + R^2\,dS^2$ \\
    Sphere of radius $a$     & $ds^2 = dR^2 + \cos^2(R/a)\,dS^2$ \\
    Hyperbolic plane         & $ds^2 = dR^2 + \sinh^2(R)\,dS^2$ \\
    \hline
    Minkowski (flat)         & $ds^2 = dR^2 - dS^2$ \\
    Rindler wedge            & $ds^2 = dR^2 - R^2\,dS^2$ \\
    2D de Sitter             & $ds^2 = dR^2 - \cosh^2(R)\,dS^2$ \\
    2D anti-de Sitter        & $ds^2 = dR^2 - \sinh^2(R)\,dS^2$ \\
    Twisted case             & $ds^2 = e^{2R}(dR^2 - dS^2)$ \\
    \hline
  \end{tabular}
  \caption{Examples of pseudo-Riemannian surfaces of constant curvature.}
  \label{tab:2d-metrics-R}
\end{table}

\begin{table}[h]
  \centering
  \renewcommand{\arraystretch}{1.25}
  \begin{tabular}{ll}
    \hline
    \textbf{Model / Example} & \textbf{Metric Form} \\
    \hline
    Unit-neck catenoid    & $ds^2 = \cosh^{2}(R)\,(dR^2 + dS^2)$ \\
    Warped product        & $ds^2 = dR^2 + f(R)^{2}\,dS^2$ \\
    Power-law warp ($p\neq 0,1$) & $ds^2 = dR^2 + R^{2p}\,dS^2$ \\
    Paraboloid ($z=R^2$)  & $ds^2 = (1+R^{2})\,dR^2 + R^2\,dS^2$ \\
    Torus of revolution   &
      $ds^2 = r^2\,dR^2 + (R_0 + r\cos R)^2\,dS^2$ \\
    Ellipsoid             &
      $ds^2 = (c^2\sin^2 R + \cos^2 R)\,dR^2 + \sin^2 R\,dS^2$ \\
    \hline
    2D Schwarzschild reduction &
      $ds^2 = \left(1-\tfrac{2M}{R}\right)^{-1} dR^2
            - \left(1-\tfrac{2M}{R}\right) dS^2$ \\
    2D FRW cosmology      & $ds^2 = dR^2 - f(R)^2\,dS^2$ \\
    Witten's 2D semi-infinite cigar &
      $ds^2 = \tfrac{1}{2}(dR^2 - \tanh^{2}R\,dS^2)$ \\
    Elliptic hyperboloid  &
      $ds^2 = (c^2\sin^2 R+\cos^2 R)\,dR^2 - \sin^2 R\,dS^2$ \\
    \hline
  \end{tabular}
  \caption{Examples of pseudo-Riemannian surfaces with non-constant curvature.}
  \label{tab:2d-metrics-R-nonconst}
\end{table}

\subsection{Energy density and Euler--Lagrange equations}

The energy density of the map $u$ is
\[
e(u) = \tfrac{1}{2} g^{ij} h_{\alpha\beta}
\partial_i u^{\alpha} \partial_j u^{\beta}
= e^{-f(v,w)}
\Bigl( A(R)\, R_v R_w - \delta^{2} B(R)\, S_v S_w \Bigr),
\]
where
\[
R_v = \tfrac{1}{2}\!\left(R_x + \tfrac{1}{\epsilon} R_y\right),
\quad
R_w = \tfrac{1}{2}\!\left(R_x - \tfrac{1}{\epsilon} R_y\right),
\]
\[
S_v = \tfrac{1}{2}\!\left(S_x + \tfrac{1}{\epsilon} S_y\right),
\quad
S_w = \tfrac{1}{2}\!\left(S_x - \tfrac{1}{\epsilon} S_y\right).
\]
The total energy is
\[
E(u) = \int_{M} e(u)\, dM
= \int_{M}
\Bigl(A(R)R_v R_w - \delta^2 B(R)S_v S_w\Bigr)\, dv\,dw.
\]
We call $u$ \textit{harmonic} if it is a critical point of $E(u)$.
Because the energy is conformally invariant in dimension two, the
Euler--Lagrange equations depend only on the conformal class of $g$.

We now compute the Euler--Lagrange equations for the Lagrangian
\[
\mathcal{L} = A(R)\, R_v R_w - \delta^{2} B(R)\, S_v S_w.
\]

\begin{lemma}\label{Euler-Lagrange}
A map $u=(R,S):(M,g)\to(N,h)$ is harmonic if and only if
\begin{align*}
2A(R) R_{vw} + A'(R) R_v R_w + \delta^{2} B'(R) S_v S_w &= 0, \\
2B(R) S_{vw} + B'(R) \left(R_v S_w + R_w S_v\right) &= 0.
\end{align*}
\end{lemma}

\begin{proof}
The Lagrangian density is
\[
\mathcal L=A(R)R_vR_w-\delta^2B(R)S_vS_w.
\]

Hence
\[
\frac{\partial \mathcal L}{\partial R}
=
A'(R)R_vR_w-\delta^2B'(R)S_vS_w,
\]
\[
\frac{\partial \mathcal L}{\partial R_v}=A(R)R_w,
\qquad
\frac{\partial \mathcal L}{\partial R_w}=A(R)R_v.
\]

The Euler--Lagrange equation for $R$ becomes
\[
\partial_v(A(R)R_w)+\partial_w(A(R)R_v)
-
\left(A'(R)R_vR_w-\delta^2B'(R)S_vS_w\right)=0,
\]
which simplifies to
\[
2A(R)R_{vw}+A'(R)R_vR_w+\delta^2B'(R)S_vS_w=0.
\]

Similarly,
\[
\frac{\partial \mathcal L}{\partial S}=0,
\qquad
\frac{\partial \mathcal L}{\partial S_v}=-\delta^2B(R)S_w,
\qquad
\frac{\partial \mathcal L}{\partial S_w}=-\delta^2B(R)S_v.
\]
Thus, the Euler--Lagrange equation for $S$ becomes
\[
2B(R)S_{vw}+B'(R)(R_vS_w+R_wS_v)=0.
\]
\end{proof}

\subsection{Reduction to first integrals}

Following \cite{FotDask2}, the Euler--Lagrange system admits useful
first integrals.

\begin{proposition}\label{FandG}
If $u$ is harmonic, then there exist smooth functions
$F(v)$ and $G(w)$ such that
\begin{align}
A(R) R_v^{2} - \delta^{2} B(R) S_v^{2} &= F(v), \label{F} \\
A(R) R_w^{2} - \delta^{2} B(R) S_w^{2} &= G(w). \label{G}
\end{align}
Conversely, if \eqref{F} and \eqref{G} hold on an open set where
$\det(Ju)\neq0$, then $u$ is harmonic.
\end{proposition}

\begin{proof}
Differentiating~\eqref{F} with respect to $w$ and using the
Euler--Lagrange equations of Lemma~\ref{Euler-Lagrange} shows that
the right-hand side is indeed a function of $v$ alone, and
similarly for~\eqref{G}. Conversely, assuming~\eqref{F}
and~\eqref{G}, differentiating with respect to $w$ (resp.\ $v$) and
rearranging recovers the Euler--Lagrange system, provided the
Jacobian determinant $\det(Ju)$ is nonzero (see Remark~\ref{rem:jacobian} below for a verification of this
condition in the concrete examples).

More precisely, if $u$ is harmonic then
		\begin{align*}
			\left( A(R)R_v^2 - \delta^2 B(R)S_v^2 \right)_w 
			&= 
			R_v\left( A'(R)R_wR_v + 2A(R)R_{vw} \right)
			\\
			&\quad - \delta^2 S_v\left( B'(R)R_wS_v + 2B(R)S_{vw} \right)
			\\
			&=
			R_v\left( A'(R)R_wR_v - A'(R)R_vR_w - \delta^2 B'(R)S_vS_w \right)
			\\
			&\quad - \delta^2 S_v\left( B'(R)R_wS_v - B'(R)R_vS_w - B'(R)R_wS_v \right)
			\\
			&=
			-\delta^2 B'(R) R_v S_v S_w 
			+ \delta^2 B'(R) S_v R_v S_w 
			= 0,
		\end{align*}
		where in the second equality we applied Lemma \ref{Euler-Lagrange}.  
		Thus,
		\[
		\left( A(R)R_v^2 - \delta^2 B(R)S_v^2 \right)_w = 0,
		\]
		and we conclude \eqref{F}.  A similar computation yields \eqref{G}. 
		
		\medskip
		
		Conversely, suppose $u$ satisfies \eqref{F} and \eqref{G}.  
		Differentiating \eqref{F} with respect to $w$ gives
		\[
		A'(R)R_wR_v^2 + 2A(R)R_vR_{wv} - \delta^2B'(R)R_wS_v^2 - 2\delta^2B(R)S_vS_{vw} = 0
		\]
		and if we add and subtract the term $\delta^2B'(R)R_vS_vS_w$ we obtain the equation
		\begin{align}\label{System 1 R-S}
			R_v\left(
			2A(R)R_{vw} + A'(R)R_wR_v + \delta^2B'(R)S_vS_w
			\right)
		\end{align}
		\[
		- \delta^2S_v 
		\left(
		2B(R)S_{vw} + B'(R)R_wS_v + B'(R)R_vS_w
		\right) = 0. 
		\]
		Differentiating \eqref{G} with respect to $v$ yields
		\[
		A'(R)R_vR_w^2 + 2A(R)R_wR_{vw} - \delta^2B'(R)R_vS_w^2 - 2\delta^2B(R)S_wS_{vw} = 0
		\]
		and if we add and subtract the term $\delta^2B(R)R_wS_vS_w$ we obtain the equation
		\begin{align}\label{System 2 R-S}
			R_w 
			\left(
			A'(R)R_vR_w + 2A(R)R_{vw} + \delta^2B'(R)S_vS_w
			\right)
		\end{align}
		\[
		-\delta^2 S_w
		\left(
		2B(R)S_{vw} + B'(R)R_vS_w + B'(R)R_wS_v
		\right) = 0
		\]
		Hence, we solve the system of the equations 
		(\ref{System 1 R-S}) and (\ref{System 2 R-S}).
		Since the Jacobian of $u$ is nonzero, i.e.
		\[
		J(u)=R_v S_w - S_v R_w \neq 0,
		\]
		the above system implies exactly the Euler--Lagrange equations. Therefore, $u$ is harmonic.
\end{proof}

From this point on, we restrict attention to the subclass
corresponding to \emph{constant} first integrals:
\begin{equation}\label{F and G}
F(v) = \kappa + \tfrac{1}{\epsilon}\lambda,
\qquad
G(w) = \kappa - \tfrac{1}{\epsilon}\lambda,
\qquad \kappa,\lambda \in \mathbb{R}.
\end{equation}
This specialisation is what enables the subsequent reduction to
ODEs. In the non-constant case the system does not admit closed-form solutions in general.

\subsection{The travelling-wave ansatz}

To achieve separation of variables we employ the following ansatz:
\begin{align}
R(x,y) &= R(ay - bx) =: R(t), \label{R} \\
S(x,y) &= ax + by + H(ay - bx) =: ax + by + H(t), \label{S}
\end{align}
where $t = ay - bx$ and $a, b \in \mathbb{R}$.

\begin{remark}\label{rem:ansatz-motivation}
The variable t is a travelling-wave parameter whose level sets are affine lines in the domain.
\end{remark}

Substituting~\eqref{R}--\eqref{S} into~\eqref{F}--\eqref{G} reduces
the Euler--Lagrange system to a manageable system of ODEs for $R(t)$
and $H(t)$. The explicit solutions derived from these ODEs constitute
the harmonic maps presented in Theorem~\ref{Theorem1}.

\section{Main result and proof}
The following result should be understood as a local reduction statement, valid on regions where the induced Jacobian of the ansatz remains nondegenerate.

\begin{theorem}\label{Theorem1}
Let $(M,g)$ be a pseudo-Riemannian surface with isothermal metric
$g = e^{f(x,y)}(dx^2 - \epsilon^2 dy^2)$, $\epsilon \in \{1,i\}$,
and let $(N,h)$ be a pseudo-Riemannian surface with metric
$h = A(R)\,dR^2 - \delta^2 B(R)\,dS^2$, $\delta \in \{1,i\}$, where
$A, B$ are smooth positive functions. Let $a, b, \kappa, \lambda \in
\mathbb{R}$ be constants with $b^2 \neq \epsilon^2 a^2$ and $ab \neq 0$, and set
$t = ay - bx$. Consider the map $u = (R,S):(M,g)\to(N,h)$ defined by
\[
R(x,y) = R(t), \qquad S(x,y) = ax + by + H(t),
\]
where
\begin{align}\label{H'}
H'(t) =
\frac{2\lambda\delta^2(b^2 + \epsilon^2 a^2)
     + 4\kappa\delta^2 ab
     + ab(a^2+b^2)(\epsilon^2+1)B(R)}
     {(b^2 - \epsilon^2 a^2)(a^2+b^2)B(R)}
\end{align}
\begin{align}\label{R-improved}
R'(t)^{2}
= \frac{c_{1}B(R)^{2}
       + c_{2}\kappa\, B(R)
       + c_{3}\lambda\, B(R)
       + c_{4}}
       {(a^{2}+b^{2})^{2}(b^{2}-\epsilon^{2}a^{2})^2 A(R)B(R)},
\end{align}
and
\[
c_{1} = \delta^{2}\epsilon^{2}(a^{2}+b^{2})^{4}, \quad
c_{2} = 4(a^{2}+b^{2})^{2}(b^{2}+\epsilon^{2}a^{2}),
\]
\[
c_{3} = 8\epsilon^{2}ab(a^{2}+b^{2})^{2}, \quad
c_{4} = 4\delta^{2}\bigl(2\kappa ab
        + \lambda(b^{2}+\epsilon^{2}a^{2})\bigr)^{2}.
\]
Then $u$ is a harmonic, respectively wave, map. 
\end{theorem}
Equation~\eqref{R-improved} is separable; given any smooth solution
$R(t)$ on an interval where its right-hand side is positive,
equation~\eqref{H'} yields $H$ by quadrature, and the resulting map
$u$ is harmonic.

\begin{proof}
We substitute the ansatz \eqref{R}--\eqref{S} into the first
integrals \eqref{F}--\eqref{G} with the constants as
in~\eqref{F and G}. A
direct computation gives the system: 
\begin{align*}
(a^{2}+\epsilon^{2}b^{2})A(R)R'(t)^{2}
&= 4\epsilon^{2}\kappa
  + \delta^{2}B(R)\bigl[(b+aH'(t))^{2}
  + \epsilon^{2}(a-bH'(t))^{2}\bigr], \\
-abA(R)R'(t)^{2}
&= \delta^{2}B(R)(a-bH'(t))(b+aH'(t)) + 2\lambda. 
\end{align*}

Expanding the quadratic expressions and introducing the shorthand
\[
K := A(R)R'(t)^{2} - \delta^{2}B(R)H'(t)^{2},
\]
the system turns into the following equations
\begin{align}
(a^{2}+\epsilon^{2}b^{2})K
&= 4\epsilon^{2}\kappa
  + \delta^{2}B(R)\bigl(b^{2}+\epsilon^{2}a^{2}
  + 2(1-\epsilon^{2})abH'(t)\bigr), \label{System K1}\\
abK
&= -2\lambda
  + \delta^{2}B(R)\bigl((b^{2}-a^{2})H'(t) - ab\bigr). \label{System K2}
\end{align}

\noindent\textit{Step 1: Solve for $H'(t)$.}
Multiply \eqref{System K2} by $(a^2+\epsilon^2 b^2)/(ab)$ and
subtract from \eqref{System K1} to eliminate $K$:
\begin{align*}
\delta^2 B(R)
\left[
  \frac{(a^2+\epsilon^2 b^2)(b^2-a^2)}{ab}
  - 2(1-\epsilon^2)ab
\right] H'(t)
\\
= 4\epsilon^2\kappa
  + \frac{2\lambda(a^2+\epsilon^2 b^2)}{ab}
  - \delta^2 B(R)(b^2+\epsilon^2 a^2).
\end{align*}
Combining the bracket on the left, we obtain that
\[
\frac{(a^2+\epsilon^2 b^2)(b^2-a^2)}{ab} - 2(1-\epsilon^2)ab
= \frac{(b^2-\epsilon^2 a^2)(a^2+b^2)}{ab},
\]
and dividing through (using $b^2 \neq \epsilon^2 a^2$)
yields \eqref{H'}.

\noindent\textit{Step 2: Derive the ODE for $R(t)$.}

Let \(M\) denote the numerator of \eqref{H'}. Substituting \(H'(t)\) into \eqref{System K2}
		and expanding the resulting expression yields
		\[
		R'(t)^2 = \frac
		{\delta^2abM^2 + (b^4 - a^4)(b^2 - \epsilon^2a^2)M\delta^2B(R)}
		{ab(a^2 + b^2)^2(b^2 - \epsilon^2a^2)^2A(R)B(R)}
		\]
		\[
		- \frac{ 
			\left(
			2\lambda + ab\delta^2B(R)
			\right)
			(a^2 + b^2)^2
			(b^2 - \epsilon^2a^2)^2B(R)
		}
		{ab(a^2 + b^2)^2(b^2 - \epsilon^2a^2)^2A(R)B(R)}
		\]
		
	A direct computation of the numerator yields the coefficients for the powers of $B(R)$ as follows:
		\[
		\text{coefficient of }(B(R)^{2})=\delta^{2}\epsilon^{2}ab(a^{2}+b^{2})^{4},
		\]
		\[
		\text{coefficient of }(\kappa B(R))=4ab(a^{2}+b^{2})^{2}(b^{2}+\epsilon^{2}a^{2}),
		\]
		\[
		\text{coefficient of }(\lambda B(R))=8\epsilon^{2}a^2b^2(a^{2}+b^{2})^{2},
		\]
		\[
		\text{constant }=4\delta^{2}ab
		\big(2\kappa ab+\lambda(b^{2}+\epsilon^{2}a^{2})\big)^{2}.
		\]
		
		Thus, formula \eqref{R-improved}
is valid.

\noindent\textit{Step 3: Conclusion.}
Since the right-hand side of \eqref{R-improved} depends only on
$R(t)$, the equation is separable and can be integrated to give $R(t)$
on any interval where the right-hand side is positive. Substituting
the result into \eqref{H'} and integrating gives $H(t)$, and hence
$S(x,y)$. The map $u$ satisfies \eqref{F}--\eqref{G} by
construction, so it is harmonic by Proposition~\ref{FandG}.
\end{proof}

\begin{remark}\label{rem:jacobian}
A direct computation gives
\[
\det(Ju)=-(a^2+b^2)R'(t).
\]
Thus the nondegeneracy condition is equivalent to
\[
R'(t)\neq 0.
\]
Accordingly, the reduction argument applies locally on intervals where the profile function remains monotone. This condition is only required in the converse direction of
Proposition~\ref{FandG}. In the explicit examples below, the domain
is restricted to intervals where $R'(t)\neq0$.
\end{remark}

\begin{remark}\label{rem:degeneracy}
The assumption $b^2 \neq \epsilon^2 a^2$ is necessary for the
formulas \eqref{H'} and \eqref{R-improved} to be well-defined.
When $b^2 = \epsilon^2 a^2$, then $\epsilon=1$ and from the computation in Step 1 we find that $B(R)$ must be constant, forcing $R(t)$ to be constant and the Jacobian to vanish identically. 
Moreover if $a = 0$ or $b = 0$ then the calculations are much easier and the proof is omitted.
\end{remark}


\section{Explicit harmonic and wave map constructions}
\label{sec:examples}

We illustrate the scope of Theorem~4 through four representative examples. Our aim is not to exhaust all possible cases, but to show how the reduction produces explicit solutions for genuinely variable-curvature targets and in both elliptic and hyperbolic regimes. 

\subsection{Harmonic maps into ellipsoids}

Let $E_c$ denote the ellipsoid
\[
\frac{x^2}{c^2} + y^2 + z^2 = 1, \quad c > 1,
\]
endowed with the induced metric from $\mathbb{R}^3$. In geodesic
coordinates $(R,S)$, this metric takes the form
\[
h = (c^2\sin^2 R + \cos^2 R)\,dR^2 + \sin^2 R\,dS^2,
\]
so $A(R) = c^2\sin^2 R + \cos^2 R$ and $B(R) = \sin^2 R$.
The Gaussian curvature of $E_c$ is
\[
K(R) = \frac{c^2}{(c^2\sin^2 R + \cos^2 R)^2},
\]
which is non-constant for $c \neq 1$.

Ellipsoids provide a canonical family of deformations of the round
sphere. Unlike spheres or hyperbolic planes, they do not admit
transitive isometry groups, and many techniques of constructing explicit harmonic maps into ellipsoids, relying on symmetry
or constant curvature fail in this setting.

\begin{proposition}
		The map $u$ from a Riemannian surface M to an ellipsoid Riemannian surface $E_c$, i.e.
		\[
		u = (R,S): (M,e^{f(x,y)}(dx^2 + dy^2)) \rightarrow 
		(E_c, (c^2\sin^2R + \cos^2R)dR^2 + \sin^2R dS^2),
		\]
		where 
		\[
		R(x,y) = 
		\arccos
		\left(
		\cos\theta \sin
		\left(\frac{ay - bx}{\sqrt{c^2 - 1}}
		\right)
		\right),
		\]
		\[
		S(x,y) = 
		ax + by +
		\arctan
		\left(\sin\theta 
		\tan
		\left(\frac{ay - bx}{\sqrt{c^2 - 1}}
		\right)
		\right),
		\]
		with $\theta \in (0,\pi/2)$, is a harmonic map on the open set where the map is well defined and the Jacobian is nonvanishing.
	\end{proposition}
	\begin{proof}
		From equation (\ref{R-improved}), we obtain
		\[
		R'(t)^2 = \frac{c_2\sin^4R(t) + c_1\sin^2R(t) + c_0}
		{(a^2 + b^2)^4\sin^2R(t)((c^2 - 1)\sin^2R(t) + 1)}
		\]
		where 
		\[
		c_2 = (a^2 + b^2)^4, \quad 
		c_1 = 4(a^2 + b^2)^2
		\left(
		(b^2 - a^2)\kappa - 2ab\lambda
		\right),
		\]
		\[
		c_0 = - 4\left(
		2ab\kappa + (b^2 - a^2)\lambda
		\right)^2.
		\]
		Note that 
		\[c_{1}^2-4c_0 c_2 = 4^2 (a^2+b^2)^6 (\kappa^2 + \lambda^2) > 0.\]
		Since the discriminant of the numerator is positive, the numerator turns into $c_2(\sin^2R - R_1)(\sin^2R - R_2)$. Since $R_1 \cdot R_2 < 0$,  we can assume that $R_1 > 0$ and $R_2 < 0$ and so it follows that
		\[
		R'(t)^2 = \frac
		{(\sin^2 R(t) - R_1)(\sin^2 R(t) - R_2)}
		{\sin^2R(t)
			\left((c^2 - 1)\sin^2R(t) + 1\right)}.
		\]
		Since $c > 1$ we can set $R_2 = - \frac{1}{c^2 - 1}$ and since we want the numerator to be positive we set $R_1 = \sin^2\theta$ with $\theta \in (0,\pi/2)$. To this end, 
		\[
		R'(t)^2 = \frac{\sin^2 R(t) - \sin^2\theta}
		{(c^2 - 1)\sin^2R(t)}
		\Rightarrow
		\cos^2R(t) = \cos^2\theta\sin^2
		\left(\frac{t}{\sqrt{c^2 - 1}}\right).
		\]

		Since we have determined the function $R(t)$, we can compute $H(t)$ as follows
		\[
		H'(t) = \frac{-2\lambda(b^2 - a^2) - 4\kappa ab}
		{(a^2 + b^2)^2 \sin^2R(t)}
		\]
		and by replacing $R(t)$, we obtain that
		\[
		H'(t) = \frac{2\lambda(a^2 - b^2) - 4\kappa ab}{(a^2 + b^2)^2}
		\frac{1}{1 - \cos^2\theta \sin^2(\frac{t}{\sqrt{c^2 - 1}})}
		\]
		\[
		\Rightarrow
		H(t) = 
		\frac{2\lambda(a^2 - b^2) - 4\kappa ab}{(a^2 + b^2)^2}
		\frac{\sqrt{c^2 - 1}}{\sin\theta}
		\arctan
		\left(
		\sin\theta \tan(\frac{t}{\sqrt{c^2 - 1}})
		\right).
		\]
		Taking into account that $R_1\cdot R_2 = \frac{c_0}{c_2}$, we find 
		\[
		H(t) =  \arctan
		\left(
		\sin\theta
		\tan
		\left(
		\frac{t}{\sqrt{c^2 - 1}}
		\right)
		\right).
		\]
		To this end, we obtain $S(x,y) = ax + by + H(ay - bx)$.
		The corresponding constants $\kappa$ and $\lambda$ are determined by the following system
		\[
		(b^2 - a^2)\kappa - 2ab\lambda = 
		\frac{(a^2 + b^2)^2}{4}
		\left(
		\frac{1}{c^2 - 1} - \sin^2\theta
		\right),
		\]
		\[
		2ab\kappa + (b^2 - a^2)\lambda = 
		- 
		\frac{(a^2 + b^2)^2}{4}
		\frac{2\sin\theta}{\sqrt{c^2- 1}}
		\]
		and so we obtain
		\begin{align}\label{kappa}
			4\kappa = 
			(b^2 - a^2)
			\left(\frac{1}{c^2 - 1} - \sin^2\theta\right) - 
			\frac{4ab\sin\theta}{\sqrt{c^2 - 1}}
		\end{align}
		\begin{align}\label{lambda}
			2\lambda= 
			-(b^2 - a^2)\frac{\sin\theta}{\sqrt{c^2 - 1}}
			- ab\left(\frac{1}{c^2 - 1} - \sin^2\theta\right)
		\end{align}

		For completeness, we shall sketch the proof that the map $u$ satisfies the equations 
		\begin{align*}
			4\kappa &= A(R)(R_x^2 - R_y^2) - B(R)(S_x^2 - S_y^2),
			\\
			2\lambda &= A(R)R_xR_y + B(R)S_xS_y,
		\end{align*}
		where $\kappa, \lambda$ are given by (\ref{kappa}) and (\ref{lambda}).
		
		By the formula 
		\[
		\cos{R(t)}=\cos\theta\sin\left(\frac{t}{\sqrt{c^2-1}}\right)
		\]
		we find that 
		\[
		R'(t)^2 =\frac{1}{c^2-1}\frac{\sin^2{R(t)}-\sin^2{\theta}}{\sin^2{R(t)}}.
		\]
		Similarly, by the formula
		\[
		\tan{H(t)}=\sin{\theta}
		\tan
		\left(\frac{t}{\sqrt{c^2-1}}
		\right)
		\]
		we find that
		\[
		H'(t)=\frac{\sin{\theta}}{\sqrt{c^2-1}}\frac{1}{\sin^2{R(t)}}
		\]
		Recall now that 
		\begin{align*}
			- 4\kappa
			&=
			(a^{2}-b^{2})K
			+ \sin^2{R(t)}\big(b^{2}-a^{2}+4abH'(t)\big),\\[0.3em]
			-2\lambda
			&=
			abK
			+\sin^2{R(t)}\big( (b^{2}-a^{2})H'(t) - ab\big),
		\end{align*}
		where
		\begin{align*}
			K &= (1+(c^2-1)\sin^2{R(t)}) R'(t)^{2}+\sin^2{R(t)} H'(t)^{2}\\
			&=\sin^2{R(t)} + \frac{1}{c^2-1} - \sin^2\theta.
		\end{align*}
		Then, a straightforward calculation reveals that
		\begin{align*}
			(a^{2}-b^{2})K+\sin^2{R(t)}\big(b^{2}-a^{2}+4abH'(t)\big)
			\\=-(b^2 - a^2)
			\left(\frac{1}{c^2 - 1} + \sin^2\theta\right) +
			\frac{4ab\sin\theta}{\sqrt{c^2 - 1}} =-4\kappa
		\end{align*}
		and 
		\begin{align*}
			abK+  \sin^2{R(t)}\big( (b^{2}-a^{2})H'(t) - ab\big)
			\\=\frac{\sin\theta(b^2 - a^2)}{\sqrt{c^2 - 1}}
			+ab\left(\frac{1}{c^2 - 1} - \sin^2\theta\right)=-2\lambda.
		\end{align*}
		So we have verified that the map $u$ is harmonic and the proof is complete.
	\end{proof}

\subsection{Lorentzian wave maps into hyperboloids}

Let $F_c$ denote the elliptic hyperboloid of one sheet
\[
\frac{x^2}{c^2} + y^2 - z^2 = 1, \quad c > 1,
\]
endowed with the induced Lorentzian metric from
$\mathbb{R}^{2,1}$. In appropriate coordinates $(R,S)$ this metric
takes the form
\[
h = (c^2\sin^2 R + \cos^2 R)\,dR^2 - \sin^2 R\,dS^2,
\]
so $\delta = 1$, $A(R) = c^2\sin^2 R + \cos^2 R$ and $B(R) = \sin^2 R$.

\begin{proposition}\label{prop:hyperboloid}
The map $u$ from a Lorentzian surface $M$ to the elliptic hyperboloid Lorentzian surface $F_c$, i.e.
		\[
		u= (R,S): (M,e^{f(x,y)}(dx^2 - dy^2)) \rightarrow
		(F_c, (c^2\sin^2R + \cos^2R)dR^2 - \sin^2RdS^2)
		\]
		where 
		\[
		R(x,y) = \arccos 
		\left(
		\cosh\theta
		\sin
		\left(   
		\omega
		(ay - bx)
		\right)
		\right),
		\]
		\[
		S(x,y) = ax + by + \frac{2ab}{b^2 - a^2}(ay - bx)
		+ \arctanh
		\left(
		\sinh\theta 
		\tan
		\left(
		\omega(ay - bx)
		\right)
		\right),
		\]
		with $\theta \in \mathbb{R}$ and 
		$\omega = \frac{a^2 + b^2}{b^2 - a^2}
		\frac{1}{\sqrt{c^2 - 1}}$, is a Lorentzian wave map on the open set where the map is well defined and the Jacobian is nonvanishing.
	\end{proposition}
	\begin{proof}
	The proof follows similar steps as in the Riemannian ellipsoid case. Now we find
	\[
		R'(t)^2 = 
		\left(
		\frac{a^2 + b^2}{b^2 - a^2}
		\right)^2
		\frac{(\sin^2R - R_1)(\sin^2R - R_2)}
		{\sin^2R\left((c^2 - 1)\sin^2R + 1\right)}.
		\]
		and set $R_1 = - \frac{1}{c^2 - 1} < 0$ and
		$R_2 = - \sh^2\theta$, where $\theta \in \mathbb{R}$. So, by integrating
		\[
		\cos^2R(t) = \cosh^2\theta 
		\sin^2(\omega t),
		\]
		where $\omega  = \frac{a^2 + b^2}{(b^2 - a^2)\sqrt{c^2 - 1}}$.
		
		By substitution and integration we obtain
		\[
		H(t) =  \frac{2ab}{b^2 - a^2}t + 
		\arctanh
		\left(
		\sh\theta \tan(\omega t)
		\right).
		\]

	The values of $\kappa$ and $\lambda$ are given by the following system
	\begin{align}\label{kappa2}
			4\kappa  = \frac{(a^2 + b^2)^2}{(b^2 - a^2)^2}
			\left(
			(a^2 + b^2)
			\left(\sh^2\theta + \frac{1}{c^2 - 1}\right) - 
			\frac{4ab\sh\theta}{\sqrt{c^2 - 1}}
			\right)
		\end{align}
		\begin{align}\label{lambda2}
			2\lambda & = \frac{(a^2 + b^2)^2}{(b^2 - a^2)^2}
			\left(
			(a^2 + b^2)\frac{\sh\theta}{\sqrt{c^2 - 1}} - 
			ab\left(\sh^2\theta + \frac{1}{c^2 - 1}\right)
			\right) .
		\end{align}

	For completeness, we shall sketch the proof that the map $u$ satisfies the following equations,
		\begin{align*}
			4\kappa &= A(R)(R_x^2 + R_y^2) - B(R)(S_x^2 + S_y^2), \\
			2\lambda & = A(R)R_xR_y - B(R)S_xS_y.
		\end{align*}
		where $\kappa, \lambda$ are given by (\ref{kappa2}) and (\ref{lambda2}).
		
		By the formula 
		\[
		\cos{R(t)}=\cosh\theta\sin(\omega t)
		\]
		we find that 
		\[
		R'(t)^2 =\omega^2\frac{\sinh^2\theta + \sin^2{R(t)}}{\sin^2{R(t)}}.
		\]
		Similarly, by the formula
		\[
		H(t) =  \frac{2ab}{b^2 - a^2}t + 
		\arctanh
		\left(
		\sh\theta \tan(\omega t)
		\right).
		\]
		we find that
		\[
		H'(t)=\frac{2ab}{b^2-a^2}+\omega\frac{\sinh{\theta}}{\sin^2{R(t)}}
		\]
		
		Recall now that we have to prove that the following equations are valid,
		\begin{align*}
			(a^{2}+b^{2})(K- \sin^2{R(t)})
			&=
			4\kappa,\\[0.3em]
			abK-\sin^2{R(t)}\big( (b^{2}-a^{2})H'(t) - ab\big)
			&=
			-2\lambda,
		\end{align*}
		where
		\begin{align*}
			K &= (1+(c^2-1)\sin^2{R(t)}) R'(t)^{2} - \sin^2{R(t)} H'(t)^{2}\\
			&=\sin^2{R(t)}
			\left( \omega^2 (c^2-1)-\frac{4a^2b^2}{(b^2-a^2)^2}\right)
			\\
			&+\left(\omega^2 (1+(c^2-1)\sinh^2{\theta})-\omega\sinh{\theta} \frac{4ab}{b^2-a^2} \right) .
		\end{align*}
		Then, a straightforward calculation reveals that
		\[
		(a^{2}+b^{2})(K- \sin^2{R(t)})
		\]
		\[
		=\frac{(a^2 + b^2)^2}{(b^2 - a^2)^2}
		\left(
		(a^2 + b^2)
		\left(\sh^2\theta + \frac{1}{c^2 - 1}\right) - 
		\frac{4ab\sh\theta}{\sqrt{c^2 - 1}}
		\right) =4\kappa   
		\] 
		and 
		\[
		abK-\sin^2{R(t)}\big( (b^{2}-a^{2})H'(t) - ab\big)
		\]
		\[
		=\frac{(a^2 + b^2)^2}{(b^2 - a^2)^2}
		\left(
		ab\left(\sh^2\theta + \frac{1}{c^2 - 1}\right) -(a^2 + b^2)\frac{\sh\theta}{\sqrt{c^2 - 1}}
		\right) =-2\lambda.
		\]
		So we have verified that the map $u$ is a Lorentzian wave map and the proof is complete.
	\end{proof}

\subsection{Lorentzian wave maps into the Schwarzschild exterior}

We construct an explicit wave map whose target is the two-dimensional
Lorentzian reduction of the Schwarzschild spacetime:
\[
h = \left(1 - \frac{2M}{R}\right)^{-1}dR^2
  - \left(1 - \frac{2M}{R}\right)dS^2,
\qquad R > 2M,\; M > 0.
\]
Here $A(R) = (1-2M/R)^{-1}$ and $B(R) = 1-2M/R$, so
$A(R)B(R) = 1$. This degeneracy considerably simplifies the
reduction.

\begin{proposition}\label{prop:schwarz}
Let $M>0$ and $b > a > 0$. Set $\omega = \frac{a^2+b^2}{b^2-a^2} > 0$
and $t = ay - bx$. The map $u = (R,S)$ from a Lorentzian surface
$(M, e^{f(x,y)}(dx^2-dy^2))$ to the Schwarzschild exterior, defined by
\[
S(x,y) = \omega(by - ax),
\]
and $R(x,y) = R(t)$ is determined implicitly by
\[
R(t) + 2M\log(R(t)-2M) = \omega(t-t_0), \qquad R(t) > 2M,
\]
for some $t_0 \in \mathbb{R}$, is a wave map. The implicit relation
defines $R(t)$ as a smooth, strictly monotone function of $t$ on the
Schwarzschild exterior region $\{R > 2M\}$.
\end{proposition}

\begin{proof}
Set $\epsilon = \delta = 1$ and $\kappa = \lambda = 0$.
Since $A(R)B(R) = 1$,
equation~\eqref{R-improved} becomes
\[
R'(t)^2 = \omega^2\bigl(1 - \tfrac{2M}{R}\bigr)^2.
\]
Taking the positive square root, separating variables and integrating gives the stated implicit formula. Formula~\eqref{H'} gives $H'(t) = 2ab/(b^2-a^2)$,
a constant, so $S(x,y) = ax+by + \frac{2ab}{b^2-a^2}(ay-bx) =
\omega(by-ax)$. The right-hand side of~\eqref{R-improved} is
strictly positive for $R>2M$, so $R'(t)\neq 0$ and the map is
nondegenerate.
\end{proof}

\begin{remark}
The condition $R > 2M$ restricts the wave map to the Schwarzschild
exterior, avoiding the event horizon. The implicit relation
$R + 2M\log(R-2M) = \omega(t-t_0)$ is smoothly invertible on this
region since its left-hand side is strictly increasing in $R$.
This example illustrates that the framework applies naturally to
Lorentzian geometries arising in general relativity.
\end{remark}

\subsection{A mixed-signature example}

We now consider the case $\delta \neq \epsilon$, so that
$\delta^2\epsilon^2 = -1$.

\begin{proposition}
		The map $u=(R,S)$ from pseudo-Riemannian surface $M$ that is equipped with the metric $dx^2-\epsilon^2 dy^2$ to the pseudo-Riemannian surface $N$ that is equipped with the metric  
		\[
		h=dR^2 -\delta^2\tanh^2R dS^2,
		\]
		where 
		\[
		R(x,y) = \operatorname{arcsinh}
		\left(
		\frac
		{\cos{\theta}\cdot y-\sin{\theta} \cdot x}
		{\cos^2{\theta} - \epsilon^2\sin^2{\theta} }
		\right),
		\]
		\[
		S(x,y) = \cos{\theta}\cdot x+\sin{\theta}\cdot y+\frac{(1+\epsilon^2)\cos{\theta}\sin{\theta}}
		{(\cos^2{\theta} - \epsilon^2\sin^2{\theta}) }
		(\cos{\theta} \cdot y-\sin{\theta}\cdot x),
		\]
		with $\theta \in (0,\frac{\pi}{4})$, is a wave map.
	\end{proposition}
	This example illustrates the flexibility of the method and confirms that the reduction mechanism is insensitive to the relative signatures of the domain and target. The above metric provides another example of a variable-curvature surface.
	
	\begin{proof}
For $a=\cos{\theta}, b=\sin{\theta}, A(R)=1, B(R)=\tanh^2{R}$ and considering $\kappa, \lambda$ given by the system
	\begin{align}\label{kappa3}
		(b^2+\epsilon^2 a^2)\kappa+2\epsilon^2 ab\lambda &=1/4
	\end{align}
	\begin{align}\label{lambda3}
		2ab\kappa+(b^2+\epsilon^2 a^2)\lambda =0
	\end{align}
	we derive the equations
	\[
	R'(t)^2 = 
	\frac{(1-\tanh^2{R})}{(\sin^2{\theta}-\epsilon^2 \cos^2{\theta})^2}
	\]
	and thus
	\[
	H'(t)=\frac{(1+\epsilon^2)\cos{\theta}\sin{\theta}}{(\sin^2{\theta}-\epsilon^2 \cos^2{\theta})}
	\]
	thus we conclude that the given map $u$ is indeed a harmonic map.

	It is straightforward to prove that the map $u$ satisfies the system
	\begin{align*}
		(a^{2}+\epsilon^{2}b^{2})K
		&=
		4\epsilon^{2}\kappa
		+\delta^{2}\tanh^2{R(t)}\big(b^{2}+\epsilon^{2}a^{2}+2(1-\epsilon^{2})abH'(t)\big),\\[0.3em]
		abK
		&=
		-2\lambda
		+\delta^{2} \tanh^2{R(t)}\big( (b^{2}-a^{2})H'(t) - ab\big),
	\end{align*}
	where
	\[
	K = R'(t)^{2}-\delta^{2}\tanh^2{R(t)}H'(t)^{2}.
	\]
	and $\kappa, \lambda$ are given by (\ref{kappa3}) and (\ref{lambda3}).
\end{proof}

\section{Concluding remarks}

We summarise the scope of the framework and indicate two natural
directions for further work.

The results of this paper fill a gap in the literature by providing 
explicit harmonic and wave maps into variable-curvature 
pseudo-Riemannian surfaces. The examples presented---including 
targets of both Riemannian ($\delta = i$) and Lorentzian 
($\delta = 1$) signature---verify the construction and yield 
closed-form solutions that were previously inaccessible, with 
potential applications in both geometric analysis and mathematical 
physics.

In Proposition~\ref{FandG}, allowing $F(v)$ and $G(w)$ to be
non-constant functions of $v$ and $w$ respectively leads to a
non-autonomous system \eqref{System K1}--\eqref{System K2}. In that
regime, explicit closed-form solutions are generally obstructed, but
the reduction still provides a framework for analysing qualitative
properties via dynamical systems or perturbation methods.

The conformal invariance exploited here is specific to two
dimensions. Nevertheless, the core reduction mechanism, replacing
a PDE system by an algebraic constraint plus a separable ODE, can
in principle be adapted to higher dimensions. This is a direction we plan to pursue in future
work.


\vspace{20pt}

\address{
\noindent\textsc{Anestis Fotiadis}
\href{mailto:fotiadisanestis@math.auth.gr}
{fotiadisanestis@math.auth.gr}\\
Department of Mathematics, Aristotle University of Thessaloniki,\\
Thessaloniki 54124, Greece}

\vspace{10pt}

\address{
\noindent\textsc{Giannis Polychrou}
\href{mailto:polychrou.ioannis@ucy.ac.cy}
{polychrou.ioannis@ucy.ac.cy}\\
University of Cyprus, Department of Mathematics and Statistics,\\
Nicosia 1678, Cyprus}

\end{document}